\documentclass{au}
\usepackage{amssymb,latexsym}
\usepackage{amsmath}

\newcommand{\bgq}{\boldsymbol{\theta}}
\newcommand{\bgr}{\boldsymbol{\rho}}
\newcommand{\bga}{\boldsymbol{\alpha}}
\newcommand{\bgg}{\boldsymbol{\gamma}}
\newcommand{\bgt}{\boldsymbol{\tau}}
\newcommand{\gf}{\varphi}
\newcommand{\jj}{\vee}
\newcommand{\ci}{\subseteq}
\newcommand{\setm}[2]{\{\,#1\mid#2\,\}}
\newcommand{\set}[1]{\{#1\}}
\def\cng#1=#2(#3){#1\equiv#2\pod{#3}}

\newcommand{\tbf}{\textbf}

\newtheorem{theorem}{Theorem}

\begin{document}
\title{Lattice tolerances and congruences}

\author[G.\ Cz\'edli]{G\'abor Cz\'edli}
\email{czedli@math.u-szeged.hu}
\urladdr{http://www.math.u-szeged.hu/$\sim$czedli/}
\address{University of Szeged\\ Bolyai Institute\\Szeged,
Aradi v\'ertan\'uk tere 1\\ Hungary 6720}
\author[G.\ Gr\"atzer]{George Gr\"atzer}
\email{gratzer@me.com} 
\urladdr{http://server.math.umanitoba.ca/homepages/gratzer/}
\address{Department of Mathematics\\University of Manitoba\\Winnipeg, MB R3T 2N2\\Canada}

\thanks{This research was supported by the NFSR of Hungary (OTKA), 
grant no. K77432} 


\date{\today}

\begin{abstract} 
We prove that a tolerance relation of a lattice 
is a homomorphic image of a congruence relation.
\end{abstract}
\maketitle

Let $K$ and $L$ be lattices. Let  $\gf\colon K \to L$ 
be a homomorphism of $K$ onto $L$. 
Let~$\bgq$ be a congruence relation on $K$. 
Then we can define a binary relation $\gf(\bgq)$ in the obvious way:
\[
   \gf(\bgq)=\bigl\{\bigl(\gf(x),\gf(y)\bigr) \mid (x,y)\in\bgq\bigr\}.
\]

It belongs to the folklore 
(see, for instance, E. Fried and G.~Gr\"atzer
\cite{FG90} or G.~Gr\"atzer \cite{LTF}), 
that $\gf(\bgq)$ is a tolerance relation, that is, 
$\gf(\bgq)$ is a binary relation on $L$ with the following properties: 
reflexivity, symmetry, and the Substitution Properties. 

We prove the converse.

\begin{theorem}\label{T:main}
Let $\bgr$ be a tolerance relation of a lattice $L$. 
Then there exists a lattice~$K$, a congruence $\bgq$ of $K$, 
and a lattice homomorphism $\gf\colon K\to L$ such that $\bgr = \gf(\bgq)$.
\end{theorem}

\begin{proof} By a \emph{block} of $\bgr$ we mean 
a maximal subset $X$ of $L$ such that $X^2 \ci \bgr$. 
According to  \cite{cG82}, the set of all blocks of $\bgr$ 
form a lattice $L/\bgr$ such that, for $A,B\in S$, 
the join $A\vee B$ is the \emph{unique} block of $\bgr$ that includes the set
\[
   \setm{a\jj b}{a\in A,\ b\in B},
\]
and similarly for the meet. This allows us to define a \emph{lattice} $K$:
\[
    \setm{(A,x)}{A\in L/\bgr,\ x\in A},
\]
with the operations 
\[
   (A,x) \jj (B,y) = (A \jj B, x \jj y),
\]
and dually. Then 
\[
   \bgq = \bigl\{ \bigl( (A,x),(B,y) \bigr) \mid A = B\bigr\}
\]
is a congruence on $K$. Clearly, $(A,x)\mapsto x$ 
defines a homomorphism $\gf$ of $K$ onto~$L$. 
From Zorn's Lemma, we infer that $(x,y)\in \bgr$ 
if{}f $\set{x,y}\ci A$, for some $A\in L/\bgr$. 
Hence $\gf(\bgq)=\bgr$.
\end{proof}

Note that the lattice $K$ in the proof is 
the \emph{sum} of the lattices $A$, 
for $A\in L/\bgr$, in the sense of 
E.\ Graczy\'nska and G.\ Gr\"atzer \cite{GG80}; 
for the finite case this is quite easy to see via G. Cz\'edli \cite{gC09}.

For a second look at Theorem~\ref{T:main},  
let $\gf$ be a homomorphism from the lattice~$L$ 
onto the lattice $K$, with congruence kernel $\bgg$, 
so $K \equiv L /\bgg$. Let $\bga$ be a congruence on~$L$. 
We define a binary relation $\bgr$ on $K$ as follows: 
$\cng x = y (\bgr)$ if there are elements $u, v \in L$ 
such that $\cng x = y (\bga)$ and $\gf(x) = u$, $\gf(y) = v$. 
We~use the notation: $\bgr = \bga/\bgg$.

Now we rephrase Theorem \ref{T:main} with this notation.

\begin{theorem}\label{T:mainvar}
Let $K$ and $L$ be lattices and let $K = L /\bgg$. 
Then for every congruence relation $\bga$ of $L$, 
the relation $\bga/\bgg$ is a tolerance on $K$. 

Conversely, for a lattice $K$ 
and a tolerance $\bgt$ of $K$, 
there is a lattice
$L$, congruences $\bga$ and $\bgg$ on $L$, 
and a lattice isomorphism  $\gf\colon L/\bgg\to K$ 
such that $\bgt =\gf(\bga/\bgg)$.
\end{theorem}

\end{document}